\documentclass[a4paper]{amsart}

\usepackage{amsthm}
\usepackage{amsmath}
\usepackage{amssymb}
\usepackage{latexsym}
\usepackage{enumerate}
\usepackage{graphicx}
\usepackage[utf8]{inputenc}
\usepackage{epsfig} 
\usepackage{pgf}
\usepackage{tikz}
\usepackage{float}
\usepackage{dsfont}
\usepackage{times}

\renewcommand{\thetheoremName}

\newtheorem*{thmain}{Main Theorem}

\newtheorem{proposition[[]]}[theoremName]{Proposition G}

\newtheorem{theorem}{Theorem}[section]
\newtheorem{lemma}[theorem]{Lemma}

\newtheorem{proposition}[theorem]{Proposition}
\newtheorem{corollary}[theorem]{Corollary}

\theoremstyle{definition}
\newtheorem{definition}[theorem]{Definition}
\newtheorem{example}[theorem]{Example}

\newtheorem{remark}{Remark}

\numberwithin{equation}{section}

\newcommand{\Hess}{\operatorname{Hess}}
\newcommand{\dist}{\operatorname{dist}}
\newcommand{\Vol}{\operatorname{Vol}}

\newcommand{\erre}{\mathbb{R}}

\begin{document}

\title[Jellett-Minkowski's formula revisited]{Jellett-Minkowski's formula revisited.\\ Isoperimetric inequalities for submanifolds in an ambient manifold with bounded curvature}

\author{Vicent Gimeno}      
\address{Department of Mathematics-INIT, Universitat Jaume I, Castell\'o de la Plana, Spain                        
}
\email{gimenov@uji.es}

\thanks{Work partially supported by DGI grant MTM2010-21206-C02-02.}


\begin{abstract}
In this paper we provide an extension to the Jellett-Minkowski's formula for immersed submanifolds into  ambient manifolds which possesses a pole and radial curvatures  bounded from above or below by the radial sectional curvatures of a rotationally symmetric model space. Using this Jellett-Minkowski's generalized formula we can focus on several isoperimetric problems.  More precisely, on lower  bounds  for isoperimetric quotients of any precompact domain with smooth boundary, or on the isoperimetric profile of the submanifold  and its modified volume. In the particular case of a model space with strictly decreasing radial curvatures, an Aleksandrov type theorem is provided.
\keywords{Jellett-Minkowski-formula \and Aleksandrov type theorem \and isoperimetric problem \and  Mean curvature}
 \subjclass{35P15}
\end{abstract}

\maketitle

\section{Introduction}\label{intro}
Given a precompact domain $\Omega\subset P$ with smooth boundary $\partial \Omega$ in a $m-$dimensional submanifold $P^m$ of the Euclidean space $\erre^n$, by the Jellett formula one obtains (see \cite{Jellet,Chavel2,Chavel-Jellet})
\begin{equation}\label{Jellet-ine}
m\text{V}(\Omega)+ \int_\Omega \langle \tau, H\rangle dV=\int_{\partial \Omega}\langle \tau, \nu\rangle dA\quad ,
 \end{equation} 
where $\tau$ is the vector position in $\erre^n$, $V(\Omega)$ is the volume of $\Omega$, $\nu$ is the unit normal vector pointed outward to $\partial \Omega$ and,  $dV$ and $dA$ are the induced Riemannian volume and area densities on $\Omega$ and $\partial \Omega$ respectively.

The so-called Minkowski formula (see for instance \cite[Formula A in Theorem 6.11]{Montiel-Ros}) follows from the above formula for the particular case of a closed  $m$-dimensional submanifold $S$ immersed in $\erre^n$ :

\begin{equation}\label{Minkowsi}
mV(S)+\int_{S}\langle \tau, H\rangle dV=0\quad.
\end{equation} 

Making use of this formula, Jellett \cite{Jellet} proved in the middle of the nineteenth century that a star-shaped constant mean curvature surface $S\subset \erre^3$ is a round sphere (see \cite{Barros} for an extension of results in this direction).  In 1956 A. D Aleksandrov \cite{Alexandrov} improved  that result asserting that any closed, embedded hypersurface
in $\erre^n$ with constant mean curvature is a round sphere.

In this paper we provide an extension of the Jellett-Minkowski's
formula (\ref{Jellet-ine}) into a more general setting. Our setting
will be a $m$-dimensional submanifold $P^m$ immersed in a
$n$-dimensional ambient manifold $(N,o)$ with a pole $o$ and radial
curvatures $K_N$ (see \S \ref{Prelim}) bounded from above or below by
the radial sectional curvatures of a rotationally symmetric model
space $M_w^n=\erre^+\times \mathbb{S}_1^{n-1}\cup \{o_w\}$, with center point $o_w$ and warped metric
tensor $g_{M_w^n}$ constructed using the positive and increasing
warping function $w: \erre^+\to \erre^+$ in such a way that
$g_{M_w^n}=dr^2+w(r)^2g_{\mathbb{S}_1^{n-1}}$ (see \S \ref{secModelSpaces} for
precise definition and for the conditions that $w$ should attain in order to $M_w^n$ have smooth
metric tensor $g_{M^n_w}$ around $r=0$). By the Jellet-Minkowski generalized
formula we can obtain -among other results- an Aleksandrov type theorem for rotationally symmetric model manifolds.

The first step to generalize the Jellett-Minkowski's formula is to define a \emph{generalized  $w$-vector position}  $\tau_w$  using the  gradient $\nabla^N r$ of the distance function $r$ to the pole $o$ in the ambient manifold $N$
\begin{equation}\label{vector-pos}
\tau_w:=\frac{w}{w'}\nabla^Nr\quad .
\end{equation}
Our second steep is to define a \emph{weighted densities} $d\mu_w$ and $d\sigma_w$ on the submanifold using the warping function $w$ and the induced Riemannian densities $dV$ and  $dA$
\begin{equation}\label{densities}
\begin{aligned}
d\mu_w&:=w'(r) dV\quad,\\
d\sigma_w&:=w'(r) dA\quad.
\end{aligned}
\end{equation} 

Using the weighted densities $d\mu_w$ and $d\sigma_w$ for any domain  $\Omega$ with smooth boundary $\partial \Omega$ we can define the $w$-\emph{weighted volume} $\mu_w(\Omega)$ and the $w$-\emph{weighted area} ${\sigma_w}(\partial \Omega)$

 \begin{equation}\label{volumes}
\begin{aligned}
\mu_w(\Omega) &:=\int_{\Omega}d\mu_w\quad,\\
\sigma_w(\partial\Omega)&:=\int_{\partial\Omega}d\sigma_w\quad.
\end{aligned}
\end{equation}

With these previous definitions we can state following Jellett-Minkowski's generalized formula

\begin{thmain}[Jellett-Minkowski's generalized formula]\label{Jellet-teo}Let $\varphi: P^m\to N^n$ be an immersion into a $n-$dimensional ambient manifold $N$ which possesses a pole and its radial sectional curvatures $K_N$ at any point $p\in N$ are bounded by above (or below ) by the radial curvatures $K_w$ of a model space $M_w^n$ 
 $$
 K_N\left(p\right) \leq K_{M_w^n}\left(r\left(p\right)\right)=-\frac{w''}{w}\left(r\left(p\right)\right)\quad \left(\text{respectively } K_N\left(p\right)\geq -\frac{w''}{w}\left(r\left(p\right)\right)\right)\quad.
  $$ 
  Suppose moreover, that $w'>0$ . Then for any  precompact domain $\Omega$ with smooth boundary $\partial \Omega$
\begin{equation}\label{jellet-inequality}
\begin{aligned}
m \mu_w(\Omega)+\int_{\Omega}\langle \tau_w, H\rangle d\mu_w\leq (\geq) \int_{\partial \Omega}\langle \tau_w, \nu \rangle d\sigma_w \quad,
\end{aligned}
\end{equation}
where $\tau_w$ is the generalized $w$-vector position,  $\mu_w(\Omega)$ is the $w$-weighted volume of $\Omega$, $d\mu_w$ and $d\sigma_w$ are the $w$-weighted densities, and $\nu$ is the unit normal vector to $\partial \Omega$. 
\end{thmain}

\begin{remark}
Observe since $w$ is a positive increasing function expressions  (\ref{vector-pos}), (\ref{densities}) (\ref{volumes})  are well defined. In the particular case  when $w(r)=r$ one obtains 
\begin{equation}
\begin{aligned}
d\mu_w&=dV \quad \mu_w(\Omega)=V(\Omega)\\
d\sigma_w&=dA \quad \sigma_w(\partial \Omega)=A(\partial \Omega)\quad,
\end{aligned}
\end{equation}
and (\ref{jellet-inequality}) in the above theorem becomes the same expression than in (\ref{Jellet-ine}) but instead of an equality, one obtains an inequality.\end{remark}

\begin{remark}
Recall that given an isometric immersion $\varphi:(P,g_P)\to  (N,g_N)$ and the Levi-Civita connections $\nabla^N$ , $\nabla^P$ over $N$ and $P$ respectively, the second fundamental  form $B^P(X,Y)$  for any two vector fields $X$ and $Y$ of $P$ is 
$$
B^P(X,Y)=\nabla^N_XY-\nabla^P_XY\quad.
$$
In this paper we use the following convention  to define the mean curvature vector
$$
H:=\text{tr}_g(B^P)\quad.
$$
Observe that we do not divide by the dimension of the submanifold.
 \end{remark}

\begin{remark}
In the particular case when the ambient manifold is a model space inequality (\ref{jellet-inequality}) becomes an equality (see \cite{Brendle}) and we can rededuce the formula \cite[equation (6)]{Brendle}
\end{remark}
The structure of the paper is as follows.

In \S \ref{Applications} we explore the application of the above Jellett-Minkowski's generalized formula obtaining:
\begin{enumerate}
\item In the particular case in which the ambient manifold is a model space, we obtain an equality in the generalized Jellett-Minkowski formula. That allow us to provide a slight new Aleksandrov type theorem (theorem \ref{Alexandrov}) for rotationally symmetric model spaces in the line of \cite{Brendle} and to study its  application to the isoperimetric problem.
\item The  relation of the Jellett-Minkowski formula to the $k$-isoperimetric quotients.
\item Bounds  for the isoperimetric profile of isometric immersions into a geodesic ball in a Cartan-Hadamard in terms of  the total mean curvature of the submanifolds and its conformal type. 
\item  A gap result for the modified volume of complete non-compact manifolds into a Cartan-Hadamard ambient manifold. 
\item Finally, for  minimal immersions, in  corollary \ref{cor-minimal} - that is  an extension of  \cite[theorem 4]{Choe1992} - ,  we provide $w-$modified isoperimetric inequalities.
\end{enumerate}

In  \S \ref{Prelim}  we recall previous results due to the Hessian comparison of the extrinsic distance function in order to in \S \ref{Proof-Main},\S \ref{sec:5},\S \ref{sec:6}, \S \ref{sec:7}, \S \ref{Proof-Col1}, \S \ref{Proof-Col2}, \S \ref{Proof-Col3}, \S \ref{Proof-Col4}, \S \ref{Proof-Col5} prove the main theorem, the Aleksandrov type theorem (theorem \ref{Alexandrov}) and corollaries \ref{radius-isope-model}, \ref{cmc-model}, \ref{solv-isop}, \ref{tone-isoperimetric}, \ref{unuacol}, \ref{bola}, \ref{bola2}, \ref{cor-volume} and \ref{cor-minimal}.

\section{Applications of the Jellett-Minkowski's generalized formula}\label{Applications}

\subsection{Aleksandrov type theorem on $w-$model spaces} \label{secModelSpaces}
The Jellet-Minkowski generalized formula becomes an equality in the case of an immersion into a $w-$model space. Model spaces or also called $w-$Model spaces or rotationally symmetric model spaces, are generalized manifolds of revolution using warped products. Let us  recall here the following definition of a model space. 

\begin{definition}[See {\cite{GW, GriExp, GriBook}}]
A $w-$model space $M_{w}^{n}$ is a simply connected $n$-dimensional smooth manifold $M_w^n$ with a point $o_w\in M_w^n$ called the \emph{center point of the model space} such that $M_w^n-\{o_w\}$ is isometric to a smooth warped product with base $B^{1}
= (\,0,\, \Lambda)\,\,\subset\, \mathbb{R}$ (where $\, 0 < \Lambda
\leq  \infty$\,), fiber $F^{n-1} = S^{n-1}_{1}$ (i.e. the unit
$(n-1)-$sphere with standard metric), and positive warping function $w:\,
[\,0, \,\Lambda\, ) \to \mathbb{R}_{+}$. Namely: 
\begin{equation}\label{metric-model}
g_{M_w^n}=\pi^*\left(g_{(\,0,\, \Lambda)}\right)+(w\circ\pi)^2\sigma^*\left(g_{S^{n-1}_{1}}\right)\quad,
\end{equation}
being $\pi: M_w^n\to (\,0,\, \Lambda)$ and $\sigma: M_w^n\to S^{n-1}_{1}$ the projections onto the factors of the warped product.  
\end{definition}

 Despite of the freedom in the choose of the $w$ function in the above definition there exist   certain  restrictions around $r\to 0$. In order to attain $M_w^n$ a smooth  metric tensor around $o_w$ the positive warping function $w$ should hold the following equalities (see \cite{GW,Petersen}) :
\begin{equation}
\begin{aligned}
& w(0) = 0\quad,\\
&w'(0) = 1\quad,\\
& w^{(2k)}(0) = 0\quad,
\end{aligned}
\end{equation}
where $w^{(2k)}(r)$ are the even derivatives of $w$. 

The parameter  $\Lambda$ in the above definition is called the \emph{radius of the model space}. If
$\Lambda = \infty$, then $o_{w}$ is a pole of $M_{w}^{n}$.

The expression of the metric tensor (\ref{metric-model}) is often written as
$
g_{M_w^n}=dr^2+\left(w(r)\right)^2d\Theta^2
$,
where $d\Theta^2$ denotes the standard metric on $S^{n-1}_1$ ( $d\Theta^2=g_{S^{n-1}_1}$). The usual examples of model spaces are the real space forms 

\begin{remark}\label{propSpaceForm}
The simply connected space forms $\mathbb{K}^{m}(b)$ of constant
curvature $b$ can be constructed as  $w-$models with any given point
as center point using the warping functions
\begin{equation}
w(r) = w_{b}(r) =\begin{cases} \frac{1}{\sqrt{b}}\sin(\sqrt{b}\, r) &\text{if $b>0$}\\
\phantom{\frac{1}{\sqrt{b}}} r &\text{if $b=0$}\\
\frac{1}{\sqrt{-b}}\sinh(\sqrt{-b}\,r) &\text{if $b<0$} \quad .
\end{cases}
\end{equation}
Note that for $b > 0$ the function $w_{b}(r)$ admits a smooth
extension to  $r = \pi/\sqrt{b}$. For $\, b \leq 0\,$ any center
point is a pole.
\end{remark}

Applying now the generalized Jellett-Minkowski formula on embedded hypersurfaces $\Sigma$ in a model space bounding a domain $\Omega$ we can state 
\begin{corollary}\label{radius-isope-model}Let $M_w^n$ be a $w-$model space with positive warping function $w$ and positive derivative $w'$ on $(0,\Lambda)$ being $\Lambda$ the radius of the model space. Then for any closed, embedded, orientable hypersurface $\Sigma$ bounding a domain $\Omega$  
\begin{equation}
\frac{\sigma_w(\Sigma)}{\mu_w(\Omega)}\geq \frac{1}{n\sup_{\Sigma}(\vert \tau_w \vert)}\quad.
\end{equation}
\end{corollary}

Recall that a constant mean curvature  hypersurface is a hypersurface $\Sigma$  with constant  pointed inward mean curvature $h$ :
$$
h=-\langle H,\nu\rangle= \text{ constant on }\Sigma\quad,
$$
being $\nu$ the unit normal outward vector to $\Sigma$.  

For constant mean curvatures hypersurfaces using the generalized Jellett-Minkowski formula we obtain

\begin{corollary}\label{cmc-model}Let  $M_w^n$ be a $w-$model space with positive warping function $w$ and positive derivative $w'$ on $(0,\Lambda)$ being $\Lambda$ the radius of the model space. Then for any closed, embedded, orientable hypersurface $\Sigma$ with constant mean curvature  bounding a domain $\Omega$  
\begin{equation}
h=\frac{(n-1)}{n}\frac{\sigma_w(\Sigma)}{\mu_w(\Omega)}\quad .
\end{equation}
\end{corollary}

\begin{remark}
Observe that the above stated constant mean curvature hypersurfaces are mean convex hypersurfaces ($h>0$).
\end{remark}

Given a point $p\in M_w^n-\{o_w\}$ in a model space $M_w^n$ the distance $r$ to the center point $o_w$ is given by $r(p)=\pi(p)$. A $2$-plane $\Pi_{\text{radial}}$ in $T_pM_w^n$ is a \emph{radial plane} if it contains the radial direction ($\nabla r\in \Pi_{\text{radial}}$) and a $2$-plane $\Pi_{\text{fiber}}$ is a \emph{tangent to the fiber plane} if it is perpendicular to the radial direction ($\Pi_{\text{fiber}}\perp \nabla r$).
The sectional curvatures of the radial or tangent to the fibers planes depends only on the distance to the center of the model space (see proposition \ref{propModelRadialCurv}).

For model spaces with an accurate control of the sectional curvatures of the  radial and tangent to the fiber planes. We obtain the following Aleksandrov type theorem

\begin{theorem}[From Theorem 1.4 of \cite{Brendle}]\label{brendle-theorem}
Let  $(M_w^n,o_w)$ be a $w-$model space with center $o_w$, positive warping function $w$ and positive derivative $w'$ on $(0,\Lambda)$ being $\Lambda$ the radius of the model space. Suppose moreover that the sectional curvatures $K_{\Pi_{\text{fiber}}}$ of planes tangents to the fiber and the radial sectional currvatures $K_{\Pi_{\text{rad}}}$ satisfies the following relation
\begin{equation}\label{ine:secfiber-secrad}
K_{\Pi_{\text{fiber}}}\geq K_{\Pi_{\text{rad}}}.
\end{equation}
Then, if the radial sectional curvature is a monotone function non increasing on the distance to the center of the model ($\frac{d}{dr}K_{\Pi_{\text{rad}}}\leq 0$),  every closed, embedded, orientable  hypersurface $\Sigma$ with constant mean curvature  bounding a domain $\Omega$ is an umbilic submanifold. In the particular case when $K_{\Pi_{\text{fiber}}}> K_{\Pi_{\text{rad}}}$ and $\frac{d}{dr}K_{\Pi_{\text{rad}}}\leq 0$,  every closed, embedded, orientable  hypersurface $\Sigma$ with constant mean curvature  bounding a domain $\Omega$ is a sphere centered at $o_w$.
\end{theorem}

\begin{proof}
By proposition \ref{propModelRadialCurv} 
\begin{equation}
\frac{d}{dr}K_{\Pi_{\text{rad}}}=\left(-\frac{w''}{w}\right)'=\frac{w'}{w}\left(-\frac{w'''}{w'}+\frac{w''}{w}\right).
\end{equation}
Assuming the positivity of $w'$, we obtain that inequality (\ref{ine:secfiber-secrad})  and the condition $\frac{d}{dr}K_{\Pi_{\text{rad}}}\leq 0$ are attained if and only if the warping function satisfies
\begin{equation}\label{w-inequalities}
\frac{1}{w^2}-\left(\frac{w'}{w}\right)^2\geq -\frac{w''}{w}\geq-\frac{w'''}{w'} \quad.
\end{equation}
And that implies the condition H3' of \cite[Theorem 1.4]{Brendle}. Finally, the condition 
\begin{equation}
K_{\Pi_{\text{fiber}}}> K_{\Pi_{\text{rad}}}
\end{equation}
 is equivalent to condition H4'.
\end{proof}

Following the proof of the Theorem 1.4 of \cite{Brendle} we can state the following sligth modifiqued theorem

\begin{theorem}[Aleksandrov type theorem]\label{Alexandrov}
Let  $(M_w^n,o_w)$ be a $w-$model space with center $o_w$, positive warping function $w$ and positive derivative $w'$ on $(0,\Lambda)$ being $\Lambda$ the radius of the model space. Suppose moreover that the sectional curvatures $K_{\Pi_{\text{fiber}}}$ of planes tangents to the fiber and the radial sectional currvatures $K_{\Pi_{\text{rad}}}$ satisfies the following relation
\begin{equation}
K_{\Pi_{\text{fiber}}}\geq K_{\Pi_{\text{rad}}}.
\end{equation}
Then: 
\begin{enumerate}
\item If the radial sectional curvature is a monotone function non increasing on the distance to the center of the model ($\frac{d}{dr}K_{\Pi_{\text{rad}}}\leq 0$),  every closed, embedded, orientable  hypersurface $\Sigma$ with constant mean curvature  bounding a domain $\Omega$ is an umbilic submanifold. In the particular case when $K_{\Pi_{\text{fiber}}}> K_{\Pi_{\text{rad}}}$ and $\frac{d}{dr}K_{\Pi_{\text{rad}}}\leq 0$,  every closed, embedded, orientable  hypersurface $\Sigma$ with constant mean curvature  bounding a domain $\Omega$ is a sphere centered at $o_w$.
\item  If the radial sectional curvature is a  strictly decreasing function on the distance to the center of the model ($\frac{d}{dr}K_{\Pi_{\text{rad}}}<0$),  every closed, embedded, orientable  hypersurface $\Sigma$ with constant mean curvature  bounding a domain $\Omega$ is a sphere centered at $o_w$.
\end{enumerate}
\end{theorem}

\begin{remark}
Observe that in the above theorem the new item (2) ($\frac{d}{dr}K_{\Pi_{\text{rad}}}<0$) implies (by proposition  \ref{propModelRadialCurv})
\begin{equation}\label{w-inequalitiesII}
\frac{1}{w^2}-\left(\frac{w'}{w}\right)^2\geq -\frac{w''}{w}>-\frac{w'''}{w'} \quad. 
\end{equation}
\end{remark}



An example of a model space satisfying the Hypotheses of the above theorem is a generalized paraboloid 
\begin{example}[Generalized paraboloids] The usual parabolid $P$ in $\erre^3$
\[
P=\{(x,y,z)\in\erre^3 \,\vert \,z=x^2+y^2\}
\]
can be obtained as a surface of revolution with metric tensor
\[
g_P=dr^2+\left(\frac{1}{2}\text{arcsinh}(2r)\right)^2d\theta^2.
\]
In this way, we define the generalized paraboloid as a $n-$dimensional model space $M_w^n$ with warping function
\[
w(r)=\frac{1}{2}\text{arcsinh}(2r).
\]
After a few calculation one can check that
\begin{equation}
\frac{1}{w^2}-\left(\frac{w'}{w}\right)^2\geq -\frac{w''}{w}>-\frac{w'''}{w'} \quad.
\end{equation}
Hence, the closed, embedded, orientable and  constant mean curvature hypersurfaces in the generalized paraboloid are spheres centered at the center of the model.
\end{example}

The above theorem allow us to focus on the isoperimetric problem in such model spaces. The isoperimetric problem in a Riemannian manifold $M$ consists in studying, among the compact hypersurfaces $\Sigma\subset M$ enclosing a region $\Omega$ of volume $\text{V}(\Omega)=v$, those which minimize the area $\text{A}(\Sigma)$.  For a given volume $v$, an \emph{isoperimetric region} is a region of volume $v$ and with minimum area on the boundary. Using \cite[Theorem 1.1]{Ros-Iso} we know that if the $M^n$ is compact the isoperimetric problem has solution and moreover for $n\leq 7$ the isoperimetric hypersurface (the boundary of the isoperimetric region) is smooth and with constant mean curvature. Therefore,  

\begin{corollary}\label{solv-isop}
Let  $(M_w^n,o_w)$ be a $w-$model space with center $o_w$, dimension $n\leq 7$, positive warping function $w$ and positive derivative $w'$ on $(0,\Lambda)$ being $\Lambda$ the radius of the model space. Suppose moreover that the sectional curvatures $K_{\Pi_{\text{fiber}}}$ of planes tangents to the fiber and the radial sectional currvatures $K_{\Pi_{\text{rad}}}$ satisfies the following relation
\begin{equation}
K_{\Pi_{\text{fiber}}}\geq K_{\Pi_{\text{rad}}}.
\end{equation}
Suppose moreover that the radial sectional curvature is a  strictly decreasing  function on the distance to the center of the model ($\frac{d}{dr}K_{\Pi_{\text{rad}}}<0$) and $\Lambda<\infty$. Then the isoperimetric regions are geodesic balls centered at $o_w$.
\end{corollary}

\begin{remark}
Let us emphasize here, that by \cite[theorem 9.1]{Hugh99} the unique length-minimizing simple closed curve enclosing a given area is a circle centered at the origin for any plane with smooth, rotationally symmetric, complete metric such that the Gauss curvature is a strictly decreasing function from the origin. Hence, the above corollary is a (partial) generalization on the case of dimension greater than $2$. See also \cite[corollary 2.4]{Hub-Mor-2002} and \cite[corollary 2.3]{Mau-Mor-2009}.
\end{remark}

Given a domain $\Omega$ of volume $\text{V}(\Omega)$,  denoting by $\omega_{n-1}$ the volume of the standard sphere $S^{n-1}_1$ and using the following function $\text{Rad}:\erre^+\to \erre^+$ 

\begin{equation}
v\to \text{Rad}(v)\quad \text{such that}\quad v=\omega_{n-1}\int_0^{\text{Rad}(v)}w^{n-1}(t)dt,
\end{equation}
we obtain by the expression of the volume of a geodesic ball $B^w_R$ centered at $o_w$ (see equation (\ref{volume-ball}))
\begin{equation}
 \text{V}(\Omega)=\text{V}(B_{\text{Rad}(\text{V}(\Omega))}^w).
\end{equation} 

 Therefore, taking into account  and the area of a geodesic sphere $S^w_R$ of radius $R$ centered at $o_w$ in a model space (see equation (\ref{volume-sphere}) ), under the hypotheses of corollary \ref{solv-isop} for any hypersurface $\Sigma$ bounding a domain $\Omega$ of volume $V(\Omega)=v$, we have the following isoperimetric inequality:
\begin{equation}\label{isoperime-ine-model}
A(\Sigma)\geq \omega_{n-1} w(\text{Rad}(v))^{n-1}=\text{A}(S^w_{\text{Rad}(v)}).
\end{equation}

Isoperimetric inequalities have several applications, among them, the above inequality allow us to obtain an isoperimetric inequality for the first eigenvalue of the Laplacian for the Dirichlet problem (see \cite{Chavel} for details).
\begin{corollary}\label{tone-isoperimetric}
Under the hypotheses of corollary \ref{solv-isop} suppose moreover that the function 
\begin{equation}
\mathcal{I}:\erre^+\to \erre^+,\quad t\to\mathcal{I}(t)= \frac{\omega_{n-1}w(\text{Rad}(t))^{n-1}}{t}
\end{equation}
is an non-increasing function. Then, for any precompact domain $\Omega\subset M^n_w$ with smooth boundary $\partial \Omega$ and volume $\text{V}(\Omega)=v$, the first eigenvalue of the Laplacian for the Dirichlet problem $\lambda_1(\Omega)$ is bounded from below by 
\begin{equation}
\lambda_1(\Omega)\geq \frac{1}{4}\left[\mathcal{I}\left(v\right)\right]^2.
\end{equation}
\end{corollary}
\subsection{Isoperimetric inequalities and the Jellett-Minkowski formula}
In the setting of  a controlled radial  curvature  ambient manifold with a pole, and certain restrictions on the mean curvature of the submanifold,  the generalized Jellett-Minkowski formula is enough to get isoperimetric inequalities (lower bounds to the isoperimetric profile)  or to characterize the $k$-isoperimetric quotient of a domain in the submanifold. The $k-$isoperimetric quotient is defined as follows
\begin{definition}[See also \cite{Chavel-mod-constants, Chavel2}] Given a domain $\Omega$ with smooth boundary $\partial \Omega$, the \emph{$w$-weighted $k$-isoperimetric quotient} $\mathcal{I}^w_k(\Omega)$ is given by  
\begin{equation}
\mathcal{I}^w_k(\Omega):=\frac{\sigma_w(\partial \Omega)}{\mu_w(\Omega)^{\frac{k-1}{k}}} \quad,
\end{equation}
where $\mu_w(\Omega)$ and $\sigma_w(\partial \Omega)$ are the $w$-weighted volume and  $w$-weighted area respectively.
\end{definition}

To recall the definition of  the isoperimetric profile in a Riemannian manifold $M$, let us denote by $\mathcal{O}_M$ the set of relatively compact open subsets of $M$ with smooth boundary. The isoperimetric profile $I_M$ is a function $I_M:[0,\Vol(M)]\to \erre^+$ such that 
\begin{equation}
I_M(v):=\begin{cases}
0 \text{ if } v=0\\
\inf_{\Omega\in \mathcal{O}_M}\{\text{A}(\partial \Omega)\,  : \,\text{V}(\Omega)=v \}
\end{cases}
\end{equation}

In this subsection we will examine under appropriate settings the relation between  the Jellett-Minkowski's generalized formula and the $k-$isoperimetric quotients, the isoperimetric profile, and the modified volume of a submanifold.

\subsubsection{Lower bounds to the $k$-isoperimetric quotient of a domain}

The next application of the  generalized Jellett-Minkowski  formula  will be to obtain lower bounds to the $k$-isoperimetric profile of a precompact domain with smooth boundary. Hence, as a direct consequence of theorem \ref{Jellet-teo} and the H\"older's inequality  we can state  
\begin{corollary}\label{unuacol}Let $\varphi: P^m\to N^n$ be an immersion into a $n-$dimensional ambient manifold $N$ which possesses a pole and its radial sectional curvatures $K_N$ at any point $p\in N$ are bounded by above by the radial curvatures $K_w$ of a model space $M_w^n$ 
 $$
 K_N\left(p\right) \leq K_{w}\left(r\left(p\right)\right)=-\frac{w''}{w}\left(r\left(p\right)\right)\quad.
  $$ 
  Suppose moreover, that $w'>0$ . Then for any  precompact domain $\Omega$ with smooth boundary $\partial \Omega$
\begin{equation}
\mathcal{I}^w_k(\Omega)   \geq \frac{m}{\sup_{\Omega}\vert \tau_w \vert} \mu_w(\Omega)^{1/k}-\left(\int_\Omega\vert H \vert^k d\mu_w\right)^{1/k}\quad.
 \end{equation}
\end{corollary}

\subsubsection{Immersions into a geodesic ball. Isoperimetric profile, total mean curvature and Parabolicity}
Applying corollary \ref{unuacol} to the special setting of an immersion into a geodesic ball in a Cartan-Hadamard ambient manifold we obtain the following corollary for the isoperimetric profile.  
\begin{corollary}\label{bola}
Let $\varphi:P^m\to B_R^N(o)$ be an complete immersion into a geodesic ball $B_R^N(o)$ of radius $R$ centered in $o\in N$ in a  Cartan-Hadamart ambient manifold $N$ with  sectional curvatures $K_N\leq 0$ bounded from above. Suppose that the submanifold has finite $k$-norm of the mean curvature, namely, $\int_P \vert H \vert^k dV <\infty$,  for some $k>1$. Then the isoperimetric profile is bounded from below by
\begin{equation}
I_P(v)\geq \left(\frac{mv^{1/k}}{R}-\left(\int_P\vert H \vert^kdV\right)^{1/k}\right)v^{\frac{k-1}{k}}\quad,
\end{equation}
where $dV$ is the Riemannian density in $P$.
\end{corollary}

\begin{remark}The existence of isometric and complete immersions into a ball in in a Cartan-Hadamard is out of any doubt. In fact, by the celebrated Nash's imbedding theorem \cite{Nash},  any Riemannian $n$-manifold with a $C^\infty$ positive metric has a $C^\infty$ isometric imbedding in $\frac{1}{2}(n+1)(3n+1)$-dimensional Euclidean space, and in fact, in any small portion of this space (as for instance a geodesic ball of this space).

  The existence of immersions with finite integral of the norm of the mean curvature is also well known (at least in the limit case $H=0$).  In \cite{Morales1,Morales2} F. Mart\'in and S. Morales in -according to S.T Yau- highly non-trivial refinement of a Nadirashvili's method \cite{Nadi} constructed a complete minimal immersion of a $2$-dimensional disc into every open convex set of $\erre^3$.
\end{remark}

\begin{remark} From corollary \ref{unuacol} applying as a warping function $w(r)=r$ we can deduce that if we have a submanifold $P$ immersed in a ball $B_R^N(o)$ of radius $R$ centered at $o\in N$ in a Cartan-Hadamard ambient manifold $N$ with norm of the mean curvature vector $\vert H \vert $ bounded from above by
\begin{equation}
\vert H \vert \leq \frac{m\, \epsilon}{R}\quad,
 \end{equation} 
for $0<\epsilon <1$, the isoperimetric quotients are bounded from below by
\begin{equation}
I_k(\Omega)\geq\frac{m(1-\epsilon)}{R}\text{V}(\Omega)^{1/k}\quad,
\end{equation}
for any $k>1$. But that implies
\begin{equation}
\frac{\text{A}(\partial \Omega)}{\text{V}(\Omega)}\geq \frac{m(1-\epsilon)}{R}\quad,
\end{equation}
and therefore $P$ has positive Cheeger constant, positive fundamental tone and $P$ is Hyperbolic (see \cite{GriExp} for the relation of positive fundamental tone and non-parabolicity).

Hence, the mean curvature of an isometric immersion into a geodesic ball of a Cartan-Hadamard manifold  is closely related to the conformal type problem.
\end{remark}
The geodesic balls are the isoperimetric domains in the Euclidean space. Therefore, in the $n$-dimensional Euclidean space, the isoperimetric profile is given by
\begin{equation}
I_{\erre^n}(v)=C_nv^{\frac{n-1}{n}}\quad.
\end{equation}  
 Applying corollary \ref{bola}, proving by contradiction, we can easily state that for any immersion of $\erre^n$ into  a ball $B_R^{\erre^{n+k}}$ of $\erre^{n+k}$, $\int_{\erre^n}\vert H \vert^n dV=\infty$.  In general
\begin{corollary}\label{bola15}
Let $\varphi:P^m\to B_R^N(o)$ be an complete immersion into a geodesic ball $B_R^N(o)$ of radius $R$ centered in $o\in N$ in a  Cartan-Hadamart ambient manifold $N$ with  sectional curvatures $K_N\leq 0$ bounded from above. 

Suppose that for some $k>1$ there exist a constant $C_k$ such that the isoperimetric profile is bounded from above by

\begin{equation}
I_{P}(v)\leq C_k v^{\frac{k-1}{k}}\quad.
\end{equation}  

Then, either
 \begin{equation}
 \int_P \vert H \vert^k dV =\infty\quad .
\end{equation}
or 
\begin{equation}
\text{V}(P)\leq\left(\frac{C_k+\left(\int_P \vert H\vert^k d\text{V}\right)^{\frac{1}{k}}}{m}\right)^kR^k\quad.
\end{equation}
\end{corollary}
 
If moreover the submanifold has the  non-shrinking property (see definition below) we can state further geometric and analytic properties   
\begin{definition}[Non-shrinking property] A manifold $M$ of infinite volume has the \emph{non-shrinking property} if for any $v>0$ there exist $R_v$ such that
\begin{equation}
\inf_{x\in M} \text{V}(B_{R_v}^M(x))\geq v\quad.
\end{equation}
\end{definition}

Let us emphasize here, that there exist  well known manifolds with that property  
\begin{example}[Non-shrinking manifolds]
The most elemental example of a manifold with non-shrinking property is $\erre^n$. Since given a point $x \in \erre^n$ the volume $V(B_R(x))$ of   a geodesic ball of radius $R$ centered at $x$ is an increasing function of $R$, for any $v>0$ one can easily found $R_v$ such that $V(B_{R_v}(x))>v$, and by the homogeneity of $\erre^n$ we obtain the non-shrinking property. The same is true for the hyperbolic space $\mathbb{H}^n$,  and in general for any Cartan-Hadamard manifold.  
\end{example}

Using that non-shrinking property we can state that

\begin{corollary}\label{bola2}
Under the assumptions of corollary \ref{bola}, suppose moreover that $P$ has  non-shrinking property. Then
\begin{enumerate}
\item For any point $x\in P$ the geodesic balls of $P$, $B^P_R(x)$, of radius $R$ centered at $x$ satisfies
\begin{equation}
\liminf_{R\to \infty}R^{-k}\text{V}(B^P_R(x))>0,
\end{equation}
\item If moreover $k>2$, $P$ possesses a positive Green's function (or equivalently, it has transient Brownian motion).
\end{enumerate}
\end{corollary}

As a reverse of the above corollary we can state

\begin{corollary}
Let $\varphi: P^m\to B^N_R(0)$ be an isometric immersion of a parabolic submanifold $P$ into the ball $B_R^N(o)$ of radius $R$ centered at $o\in N$ in a Cartan-Hadamard manifold $N$.
Then:
\begin{enumerate}
\item or, for any $k>2$ 
$$
\int_P \vert H \vert^k dV =\infty\quad .
$$
\item Or, $P$ has not the non-shrinking property.
\end{enumerate}
\end{corollary}
By the above corollary any immersion of a parabolic manifold with non-positive sectional curvature into a geodesic ball of a Cartan-Hadamard manifold has infinite norm of its mean curvature vector.
\subsubsection{The volume of complete non-compact submanifolds into a Cartan-Hadamard ambient space}

In \cite{Caval2012} unifying results from \cite{Cheung1998,do2010,Frensel1996,Fu2010} is proved that given an isometric immersion $\varphi: P \to N$   of a complete non-compact manifold $P$ in a manifold $N$ with bounded geometry (i.e., $N$ has sectional curvature bounded from above and injectivity radius bounded from
below by a positive constant),  if  any end $E$ of $P$ has finite $L^p$-norm of the
mean curvature vector of $\varphi$,  $\Vert H \Vert_{L^p(E)} < \infty$, for some $m \leq p \leq \infty$ then $E$ must have infinite volume.

Note that $p\geq m$ is not a removable condition. Indeed in Example 4.3 of \cite{Caval2012}  is shown  a complete non-compact hypersurface $P^m$ in $\erre^{m+1}$ , with  $m\geq 3$, of finite volume and mean curvature vector with finite $L^p$-norm, for any $0\leq p <m-1$.

In the particular case of a $m$-dimensional submanifold in a Cartan-Hadamard ambient manifold with sectional curvatures bounded from above by a negative constant $b$, for any $p\geq 2$, and any dimension $m$ of the submanifold , we can obtain lower bounds for the $w_b$-weighted volume of the submanifold in terms of the $w_b$-weighted $L^p$-norm of the norm of the mean curvature vector, being $w_b$ the warping function given in remark \ref{propSpaceForm}, i.e.,
$$
w_b=\frac{1}{\sqrt{-b}}\sinh(\sqrt{-b}\,r)\quad.
$$ 
In such a setting we can state
\begin{corollary}\label{cor-volume}
Let $\varphi:P^m\to N$ be an immersion into a Cartan-Hadamard manifold $N$ with sectional curvatures $K_N$ bounded from above by a negative constant $K_N\leq b <0$. Then for any $p\geq 2$, 
\begin{enumerate}
\item either 
$$
\mu_{w_b}(P)^{\frac{1}{p}}\leq \frac{\Vert H \Vert _{L^p_{w_b}(P)}}{m\sqrt{-b}}\quad,
$$
\item or 
$$
\mu_{w_b}(P)=\infty\quad.
$$
\end{enumerate}
where  $\mu_{w_b}(P)$ is the $w_b$-weighted volume of $P$, and $\Vert H \Vert _{L^p_{w_b}(P)}$ is  the $w_b$-weighted $L^p$-norm of the norm of the mean curvature vector, namely
\begin{equation}
\Vert H \Vert _{L^p_{w_b}(P)}=\left(\int_{P}\vert H \vert^pd\mu_{w_b}\right)^{\frac{1}{p}}\quad.
\end{equation}
\end{corollary}

\subsubsection{Isoperimetric inequalities on  minimal submanifolds}
Given a simple close curve $C$ in the flat plane, bounding a domain $D$. Denoting by $L$ and $A$ the length of $C$ and the area of $D$ respectively, the classical isoperimetric inequality states that
$$
4\pi A\leq L^2\quad.
$$
For minimal submanifolds $P$ with smooth boundary $\partial P$ lying on a geodesic ball we can state a similar inequality using the $w$-weighted volume
\begin{corollary}\label{cor-minimal}
Let $P^m$ be a $m-$dimensional compact manifold with smooth boundary $\partial P$. Let $\varphi:P^m \to N^n$ a minimal immersion into a $n-$dimensional ambient manifold $N$ which possesses a pole $o\in N$ and its radial sectional curvatures $K_N$ at any point $p\in N$ are bounded by above by the radial curvatures $K_w$ of a model space $M_w^n$ 
 $$
 K_N\left(p\right) \leq K_{M_w^n}\left(r\left(p\right)\right)=-\frac{w''}{w}\left(r\left(p\right)\right)\quad.
  $$ 
  Suppose moreover, that $w'\geq c>0$. $\varphi^{-1}(o)\in P$ and $\varphi(\partial P)$ lies in a geodesic sphere centered at the pole $o$. Then 
\begin{equation}
\begin{aligned}
c m^m V_m\, \mu_{w}(P)^{m-1}\leq A(\partial P)^m\quad,
\end{aligned}
\end{equation}
where   $\mu_w(P)$ is the $w$-weighted volume of $P$, $A(\partial P)$ is the Riemannian area of $\partial P$  and $V_m$ is the volume of the unit ball in $\erre^m$. 
\end{corollary}

\begin{remark}
Applying the above corollary but using $w=w_b$ given in remark \ref{propSpaceForm}  we get, as a particular case, the isoperimetric inequality of theorem 4 of \cite{Choe1992}. See also \cite{Palmer} for an other sort of isoperimetric inequalities on minimal submanifolds properly immersed into a Cartan-Hadamard ambient manifold.
\end{remark}

\section{Preliminaries}\label{Prelim}
\subsection{Manifold with a pole  and extrinsic distance function}
We assume throughout the most part of the paper that $\varphi: P^m \longrightarrow N^n$ is an isometric immersion of a complete non-compact Riemannian $m$-manifold $P^m$ into a complete Riemannian manifold $N^n$ with a pole $o\in N$. Recall that a pole
is a point $o$ such that the exponential map
$$\exp_{o}\colon T_{o}N^{n} \to N^{n}$$ is a
diffeomorphism. For every $x \in N^{n}- \{o\}$ we
define $r(x) = r_o(x) = \dist_{N}(o, x)$, and this
distance is realized by the length of a unique
geodesic from $o$ to $x$, which is the {\it
radial geodesic from $o$}. We also denote by $r\vert_P$ or by $r$
the composition $r\circ \varphi: P\to \erre_{+} \cup
\{0\}$. This composition is called the
{\em{extrinsic distance function}} from $o$ in
$P^m$. The gradients of $r$ in $N$ and $r\vert_P$ in  $P$ are
denoted by $\nabla^N r$ and $\nabla^P r$,
respectively. Then we have
the following basic relation, by virtue of the identification, given any point $x\in P$, between the tangent vector fields $X \in T_xP$ and $\varphi_{*_{x}}(X) \in T_{\varphi(x)}N$
\begin{equation}\label{radiality}
\nabla^N r = \nabla^P r +(\nabla^N r)^\bot ,
\end{equation}
where $(\nabla^N r)^\bot(\varphi(x))=\nabla^\bot r(\varphi(x))$ is perpendicular to
$T_{x}P$ for all $x\in P$.

Since the manifold with a pole has a well defined radial vector field, we cab define the radial sectional curvatures.

\begin{definition}
Let $o$ be a point in a Riemannian manifold $N$
and let $x \in N-\{ o \}$. The sectional
curvature $K_{N}(\sigma_{x})$ of the two-plane
$\sigma_{x} \in T_{x}N$ is then called a
\textit{$o$-radial sectional curvature} of $N$ at
$x$ if $\sigma_{x}$ contains the tangent vector
to a minimal geodesic from $o$ to $x$. We denote
these curvatures by $K_{o, N}(\sigma_{x})$.
\end{definition}

\subsection{$w-$model spaces}
The model spaces has two different roles in this paper, the first of them is the role as an ambient manifold and the second one is as a controller of the curvature restrictions. The sectional curvatures of a model space can be explicitly obtained using the warped function $w$.

\begin{proposition}[See \cite{GW,GriExp,Oneill}] \label{propModelRadialCurv}
Let $M_{w}^{m}$ be a $w-$model with center point $o_{w}$. Then the
$o_{w}$-radial sectional curvatures of $M_{w}^{m}$ at every $x \in
\pi^{-1}(r)$ (for $\,r\, >\, 0\,$) are all identical and determined
by the radial function $\,K_{w}(r)\,$ defined as follows:
\begin{equation}
K_{p_{w} , M_{w}}(\sigma_{x}) \, = \, K_{w}(r) \, = \,
-\frac{w''(r)}{w(r)} \quad.
\end{equation}
And the sectional curvatures $K(\Pi_{S^w_r})$ of the $2-$planes $\Pi_{S^w_r}$ tangents to $S_r^w=\pi^{-1}(r)$ are equal to
\begin{equation}
K(\Pi_{S^w_r})=\frac{1-\left(w'\left(r\right)\right)^2}{w(r)}\quad.
\end{equation}
\end{proposition}

We can also explicitly calculate the mean curvature of the geodesic spheres. 

\begin{proposition}[See \cite{Oneill} p. 206]\label{propWarpMean}
Let $M_{w}^{n}$ be a $w-$model with warping function $w(r)$ and center $o_{w}$.
The distance sphere of radius $r$ and center
$o_{w}$ in $M_{w}^{n}$, denoted as $S^w_r$, is the fiber
$\pi^{-1}(r)$. This distance sphere has the following constant mean curvature
vector in $M_{w}^{n}$
\begin{equation}
H_{{\pi^{-1}}(r)} = -n\,\eta_{w}(r)\,\nabla^{M}\pi =  -n\,\eta_{w}(r)\,\nabla^{M}r
\quad ,
\end{equation}
where the mean curvature function $\eta_{w}(r)$ is defined by
\begin{equation}\label{eqWarpMean}
\eta_{w}(r)  = \frac{w'(r)}{w(r)} = \frac{d}{dr}\ln(w(r))\quad .
\end{equation}
\end{proposition}

In particular we have for the constant curvature space forms
$\mathbb{K}^{m}(b)$:
\begin{equation}\label{eqHbversusWarp}
\eta_{w_{b}}(r) = \begin{cases} \sqrt{b}\cot(\sqrt{b}\,r) &\text{if $b>0$}\\
\phantom{\sqrt{b}} 1/r &\text{if $b=0$}\\\sqrt{-b}\coth(\sqrt{-b}\,r) &\text{if
$b<0$} \quad .
\end{cases}
\end{equation}

The area of the geodesic sphere $S^w_R(o_w)$ of radius $R$ centered at $o_w$ is completely determined via $w$ by the volume of the fiber
\begin{equation}\label{volume-sphere}
A(S^w_R(o_w))=\omega_{n-1}w^{n-1}(R),
\end{equation}
And the volume of the corresponding ball $B_R^w(o_w)$, for which the fiber is the boundary
\begin{equation}\label{volume-ball}
V(B^w_R(o_w))=\omega_{n-1}\int_0^R w^{n-1}(t)dt,
\end{equation}
being $\omega_{n-1}$ in (\ref{volume-sphere}) and (\ref{volume-ball}) the volume of the standard sphere $S^{n-1}_1$.
\subsection{Hessian and Laplacian comparison}
The 2.nd order analysis of the restricted distance function $r_{|_{P}}$ defined on manifolds with a pole is  governed by the Hessian comparison (see \cite[Theorem A]{GW}). 

The Hessian of a restricted function in a submanifold and the Hessian of the function in the ambient space are related by the following proposition

\begin{proposition}\label{Hes-comp-sub}
Given an isometric immersion $\varphi:P^m\to N^n$, and given a smooth function $f:N\to \erre$, then:
\begin{equation}
\Hess^P(f\circ \varphi)(X,Y)=\Hess^Nf(X,Y)+\langle B^P(X,Y),\nabla^Nf\rangle.
\end{equation}
\end{proposition}

In the case of radial functions of a model space 
 
\begin{proposition}\label{hess-mod}let $M_{w}^{m}$ denote a $w-$model
with center $o_{w}$. Let $r:M_w^n\to \erre^+$ denote the distance function to the center  $o_w$. Then for any smooth function $F:\erre\to\erre$,
\begin{equation}
\begin{aligned}
\Hess^{M_w^n}F\circ r(X,Y)=&\left(F''\circ r-\left( F'\circ r\right)\left(\eta_w\circ r\right)\right)\langle X,\nabla r\rangle\langle Y,\nabla r\rangle\\&+\left( F'\circ r\right)\left(\eta_w\circ r\right)\left(\langle X,Y\rangle\right).
\end{aligned}
\end{equation}
\end{proposition}

Now, we can state a  comparison theorem   when one of the spaces is a model space $M^m_w$ using \cite{GW}:

    \begin{theorem}[See \cite{GW}, Theorem A]\label{thmGreW}
Let $N^{n}$ be a manifold with a pole $p$, let $M_{w}^{m}$ denote a $w-$model
with center $p_{w}$. Suppose that $m \leq n$ and that
every $p$-radial sectional curvature at $x \in
N - \{p\}$ is bounded from above (or below) by the $p_{w}$-radial sectional curvatures in
$M_{w}^{m}$ as follows:
\begin{equation}\label{eqKbound}
K_{p, N}(\sigma_{x}) \leq -\frac{w''(r)}{w(r)}\quad \left(\text{respectively }K_{p, N}(\sigma_{x}) \geq -\frac{w''(r)}{w(r)}\right)
\end{equation}
for every radial two-plane $\sigma_{x} \in T_{x}N$ at distance $r = r(x) =
\dist_{N}(p, x)$ from $p$ in $N$. Then the Hessian of the distance function in
$N$ satisfies
\begin{equation}\label{eqHess}
\begin{aligned}
\Hess^{N}(r(x))(X, X) &\geq (\leq)\Hess^{M_{w}^{m}}(r(y))(Y, Y)\\ &=
\eta_{w}(r)\left(1 - \langle \nabla^{M}r(y), Y \rangle_{M}^{2} \right) \\ &=
\eta_{w}(r)\left(1 - \langle \nabla^{N}r(x), X \rangle_{N}^{2} \right)
\end{aligned}
\end{equation}
for every unit vector $X$ in $T_{x}N$ and for every unit vector $Y$ in $T_{y}M$
with $\,r(y) = r(x) = r\,$ and $\, \langle \nabla^{M}r(y), Y \rangle_{M} =
\langle \nabla^{N}r(x), X \rangle_{N}\,$.
\end{theorem}

Hence, from proposition \ref{Hes-comp-sub}, proposition \ref{hess-mod}, and theorem \ref{thmGreW} after few calculations one obtains

\begin{corollary} \label{corLapComp}
Given an isometric immersion $\varphi:P^m\to N^n$. Suppose again that the assumptions of Theorem \ref{thmGreW} are
satisfied. Then,  for every smooth function $f(r)$ with $f'(r)
\geq 0\,\,\textrm{for all}\,\,\, r$ :
\begin{equation} \label{eqLap1}
\begin{aligned}
\Delta^{P}(f \circ r) \, \geq (\leq) \, &\left(\,  f''(r) - f'(r)\eta_{w}(r) \, \right)
 \Vert \nabla^{P} r \Vert^{2} \\ &+ mf'(r) \left(\, \eta_{w}(r) +\frac{1}{m}
\langle \, \nabla^{N}r, \, H_{P}  \, \rangle  \,  \right)  \quad ,
\end{aligned}
\end{equation}
where $H_{P}$ denotes the mean curvature vector of $P$ in $N$.
\end{corollary}

\section{Proof of the Jellett-Minkowski's generalized formula (main theorem)}\label{Proof-Main}
In order to prove the Jellett-Minkowski's generalized formula we only have to apply the divergence theorem to an appropriate function. Let us define the following function $F:\erre^{+}\to\erre^{+}$ given by  
\begin{equation}\label{order-function}
F(t):=\int_0^t w(s)ds\quad.
\end{equation}
Using corollary \ref{corLapComp}
\begin{equation}\label{laplacian-order-function}
\begin{aligned}
\Delta^PF&\geq (\leq) mw(r)\left(\eta_w(r)+\frac{1}{m}\langle \nabla^Nr,H\rangle\right)\\
=& \left(m+\langle \tau_w, H\rangle\right)w'(r)\quad.
\end{aligned}
\end{equation}
Applying the divergence theorem to the domain $\Omega \subset P$
\begin{equation}
\begin{aligned}
\int_{\partial \Omega}\langle \nabla^P F, \nu\rangle dA\geq (\leq) m \mu_w(\Omega)+\int_{\Omega}\langle \tau_w, H\rangle d\mu_w\quad.
\end{aligned}
\end{equation}
Therefore
\begin{equation}
\begin{aligned}
\int_{\partial \Omega}\langle \tau_w, \nu \rangle d\sigma_w\geq (\leq) m \mu_w(\Omega)+\int_{\Omega}\langle \tau_w, H\rangle d\mu_w\quad.
\end{aligned}
\end{equation}
And the theorem follows.
\section{Proof of corollary \ref{radius-isope-model} and corollary \ref{cmc-model}}\label{sec:5}
Since $\Sigma=\partial \Omega$ ($\partial \Sigma=\emptyset$), and $\Omega$ is totally geodesic submanifold  by the generalized Jellett-Minkowsi formula
\begin{equation}\label{equ-dos}
\begin{aligned}
n\mu_w(\Omega) &=\int_\Sigma \langle \tau_w, \nu\rangle d\sigma_w \quad,\\
(n-1)\sigma_w(\Sigma)&=-\int_\Sigma \langle \tau_w,H\rangle d\sigma_w\quad. 
\end{aligned}
\end{equation}
From the first line of the above equations 
\begin{equation}
n\mu_w(\Omega)\leq \max_{\Sigma}\vert \tau_w\vert \int_\Sigma d\sigma_w=\max_{\Sigma}\vert \tau_w\vert \sigma_w(\Sigma)\quad,
\end{equation}
And the corollary \ref{radius-isope-model}  follows. On the other hand,  by the second equality of (\ref{equ-dos}) if $\Sigma$ is a constant mean curvature hypersurface
\begin{equation}
\begin{aligned}
(n-1)\sigma_w(\Sigma)&=h\int_\Sigma \langle \tau_w,\nu\rangle d\sigma_w=h\,n\,\mu_w(\Omega)\quad. 
\end{aligned}
\end{equation}
And corollary \ref{cmc-model} is proven.

\section{Proof of the theorem \ref{Alexandrov}}\label{sec:6}
This proof follows the proof done in \cite{Brendle}. First of all, we need the following Heintze-Karcher \cite{Hein-Kar} inequality 
\begin{theorem}[Heintze-Karcher  inequality]\label{Hein-Kar}
Let  $M_w^n$ be a $w-$model space with center $o_w$, positive warping function $w$ and positive derivative $w'$ on $(0,\Lambda)$ being $\Lambda$ the radius of the model space.  Suppose moreover that the warping function satisfies the following inequalities:
\begin{equation}
\frac{1}{w^2}-\left(\frac{w'}{w}\right)^2\geq -\frac{w''}{w}\geq-\frac{w'''}{w'} \quad,
\end{equation}
then in every closed, embedded, orientable and  convex mean curvature hypersurface $\Sigma$ bounding a domain $\Omega$ the following inequality holds
\begin{equation}\label{ine-convex}
\left(n-1\right)\int_\Sigma \frac{1}{h}d\sigma_w\geq n\mu_w(\Omega)\quad.
\end{equation}
With equality in \ref{ine-convex} if $\Sigma$ is umbilic, and in the particular case when $\frac{w''}{w}\neq\frac{w'''}{w'}$, equality in \ref{ine-convex} implies that $\Sigma$ is a sphere centered at $o_w$.
\end{theorem}
From corollary \ref{cmc-model} it is clear that if the hypersurface is a constant mean hypersurface bounding a domain, the surface is mean convex and attains equality in inequality (\ref{ine-convex}). Hence, the only thing to do in order to prove theorem \ref{Alexandrov} is to prove the above theorem.

\begin{remark}
The proof of the above theorem follows from the proof of Theorem 3.5 of \cite{Brendle}. The only new piece is the condition $\frac{w''}{w}\neq\frac{w'''}{w'}$ and its rigidity consequences arising from remark \ref{equality-model}. For completeness in order to attain remark \ref{equality-model} we have to repeat part of the proof done in section 3 of \cite{Brendle}.  
\end{remark}

\begin{proof}
Since $w'>0$ we can use the following conformally modified metric $ g_c=\frac{1}{w'}g_{M_w^n}$ being $g_{M_w^n}$ the metric tensor in $M_w^n$. For each point $p\in \overline {\Omega}$, we denote by $u(p)=d_{g_c}(p,\Sigma)$ the distance to $p$ from $\Sigma$ with respect to the metric $g_c$, and we denote by $\Phi: \Sigma\times [0,\infty)\to \overline{\Omega}$ the normal exponential map with respect to $g_c$. Namely, for each point $x\in \Sigma$, the curve $t\to\Phi(x,t)$ is a geodesic with respect to $g_c$, and we have 
\begin{equation}\label{evolution}
\Phi(x,0)=x,\quad \left.\frac{\partial}{\partial t}\Phi(x,t)\right\vert_{t=0}=-w'(r(x))\nu(x).
\end{equation}

Let us define

\begin{equation}
\begin{aligned}
A&:=\left\{(x,t\in\Sigma\times[0,\infty)\vert\, u(\Phi(x,t))=t\right\}\\
A^*&:=\left\{(x,t\in\Sigma\times[0,\infty)\vert\, (x,t+\delta)\in A \text{  for some }\delta>0\right\}\\
\Sigma_t^*&=\Phi(A^*\cap(\Sigma\times \{t\}).
\end{aligned}
\end{equation}
We also denote by $h$ and $B^{\Sigma_t^*}$ the mean curvature and the second fundamental form of $\Sigma^*_t$ with respect to the metric $g_{M_w^n}$. Hence,
\begin{proposition}[See proposition 3.2 of \cite{Brendle}]
The mean curvature of $\Sigma_t^*$ is positive and satisfies the differential inequality
$$
\frac{\partial}{\partial t}\left(\frac{w'}{h}\right)\leq -\frac{1}{n-1}\left(w'\right)^2.
$$
\end{proposition}
\begin{proof}
Since $\nu=-\frac{\nabla u}{\vert \nabla u\vert}$ is the outward-pointing unit normal vector to $\Sigma_t^*$ with respect to the metric $g_{M_w^n}$, by the variation formulas, the mean curvature of $\Sigma_t^*$ satisfies the following equation
\begin{equation}\label{unua-des}
\frac{\partial }{\partial t}h=\Delta_{\Sigma^*_t}w'+\left(\text{Ricc}^{M_w^n}(\nu,\nu)+\Vert B^{\Sigma^*_t}\Vert^2\right)w'.
\end{equation}
Using proposition \ref{Hes-comp-sub} for any orthonormal basis $\{e_i\}_{i=1}^{n-1}$ of $\Sigma_{t}^*$
\begin{equation}
\begin{aligned}
\Delta_{\Sigma_t^*}w'&=\sum_{i=1}^{n-1}\Hess^{\Sigma_t^*}w'(e_i,e_i)\\
&=\sum_{i=1}^{n-1}\Hess^{M_w^n}w'(e_i,e_i)+\sum_{i=1}^{n-1}\langle B^{\Sigma_t^*}(e_i,e_i),\nabla^{M_w^n}w'\rangle\\
&=\Delta^{M_w^n}w'-\Hess^{M_w^n}w'(\nu,\nu)+\langle H,\nabla^{M_w^n}w'\rangle,
\end{aligned}
\end{equation}
applying proposition \ref{hess-mod} 
\begin{equation}
\begin{aligned}
\Delta_{\Sigma_t^*}w'=&w'''+(n-1)w''\eta_w-\Hess^{M_w^n}w'(\nu,\nu)+\langle H,\nabla^{M_w^n}w'\rangle\\
=& w'\left[\left(\frac{w'''}{w'}-\frac{w''}{w}\right)\left(1-\langle \nabla^{M_w^n} r, \nu\rangle^2\right)+\frac{w''}{w}\left(n-1\right)\right]\\
&+\langle H,\nabla^{M_w^n}w'\rangle.
\end{aligned}
\end{equation}
Taking into account that $1-\langle \nabla^{M_w^n} r, \nu\rangle^2\geq 0$ and $\frac{w'''}{w'}\geq \frac{w''}{w}$
\begin{equation}\label{lapla-sigma}
\begin{aligned}
\Delta_{\Sigma_t^*}w'\geq&w'\frac{w''}{w}(n-1)+\langle H,\nabla^{M_w^n}w'\rangle\\
=&w'\frac{w''}{w}(n-1)-h\langle \nu,\nabla^{M_w^n}w'\rangle.
\end{aligned}
\end{equation}
\begin{remark}\label{equality-model}
If we have equality in inequality (\ref{lapla-sigma}) then 
\begin{equation}
\left(\frac{w'''}{w'}-\frac{w''}{w}\right)\left(1-\langle \nabla^{M_w^n} r, \nu\rangle^2\right)=0.
\end{equation}
In the particular case when $\frac{w'''}{w'}\neq \frac{w''}{w}$ we would get
\begin{equation}\label{squaone}
\langle \nabla^{M_w^n} r, \nu\rangle^2=1
\end{equation}
\end{remark}

In order to make use of equality (\ref{unua-des}) we need estimate $\text{Ricc}^{M_w^n}(\nu,\nu)$ and $\Vert B^{\Sigma_t^+}\Vert^2$. Those estimates are taken care of in the next two lemmas

\begin{lemma}[From Proposition \ref{propModelRadialCurv}]
Let $M_w^n$ be a $w-$model space. Suppose that
\begin{equation}
\frac{1}{w^2}-\left(\frac{w'}{w}\right)^2\geq -\frac{w''}{w},
\end{equation}
then for any unit vector $\nu$,
\begin{equation}\label{ine-ricci}
\text{Ricc}^{M_w^n}(\nu,\nu)\geq-(n-1)\frac{w''}{w}.
\end{equation}
\end{lemma}
\begin{lemma}

For any hypersurface $S$, 
\begin{equation}\label{ine-mean}
h^2\leq (n-1) \Vert B^S \Vert^2
\end{equation}
with equality if and only if $S$ is umbilic.
\end{lemma}

Applying inequalities (\ref{lapla-sigma}),(\ref{ine-ricci}) and (\ref{ine-mean}) to equality (\ref{unua-des}) we get
\begin{equation}
\frac{\partial h}{\partial t}\geq -h\langle \nu,\nabla^{M_w^n}w'\rangle+w'\frac{h^2}{n-1}
\end{equation}
Therefore
\begin{equation}
\begin{aligned}
\frac{\partial}{\partial t}\left(\frac{h}{w'}\right)=&\frac{1}{w'}\frac{\partial h}{\partial t}-\frac{h}{\left(w'\right)^2}\frac{\partial w'}{\partial t}\\\geq &\frac{-h}{w'}\langle \nabla^{M_w^n}w',\nu\rangle+\frac{h^2}{n-1}-\frac{h}{\left(w'\right)^2}\frac{\partial w'}{\partial t}.
\end{aligned}
\end{equation}
By the evolution equation (\ref{evolution})
\begin{equation}
\frac{\partial w'}{\partial t}=-w'\nu(w')=-w'\langle \nabla^{M_w^n}w',\nu\rangle.
\end{equation} 
Hence,
\begin{equation}
\begin{aligned}
\frac{\partial}{\partial t}\left(\frac{h}{w'}\right)\geq &\frac{h^2}{n-1}.
\end{aligned}
\end{equation}
And the proposition follows.
\end{proof}
Following the proof of \cite[Theorem 3.5]{Brendle} consider the quantity
\begin{equation}
Q(t)=(n-1)\int_{\Sigma^*_t}\frac{w'}{h}dA.
\end{equation}
Therefore we can use a similar proposition to proposition 3.4 in \cite{Brendle}
\begin{proposition}
\begin{equation}
Q(0)-Q(\tau)\geq n\int_{u\leq \tau}w'dV.
\end{equation}
\end{proposition}
 Letting $\tau\to\infty$ in the above proposition,
\begin{equation}
\begin{aligned}
(n-1)\int_{\Sigma}\frac{1}{h}d\sigma_w&=(n-1)\int_{\Sigma^*_t}\frac{w'}{h}dA\\
&=Q(0)\geq n\int_{u<\infty}w'dV= n\int_{\Omega}w'dV\\
&=n\mu_w(\Omega).
\end{aligned}
\end{equation}
And the inequality (\ref{ine-convex}) follows. 

\begin{figure}
\begin{center}
\begin{picture}(0,0)%
\includegraphics{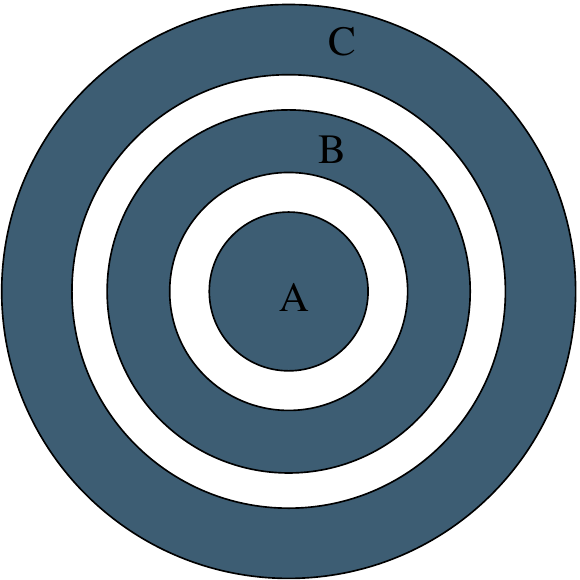}%
\end{picture}%
\setlength{\unitlength}{4144sp}%
\begingroup\makeatletter\ifx\SetFigFont\undefined%
\gdef\SetFigFont#1#2#3#4#5{%
  \reset@font\fontsize{#1}{#2pt}%
  \fontfamily{#3}\fontseries{#4}\fontshape{#5}%
  \selectfont}%
\fi\endgroup%
\begin{picture}(2640,2638)(9931,-1830)
\end{picture}%

\caption{By equality (\ref{squaone}) $\Sigma$ is a sphere or a finite union of spheres, hence $\Omega$ is either a geodesic ball centered at $o_w$ ( A ), or  a geodesic annulus centered at $o_w$ ( B ) or a finite union of annuli (B+C ) or   A geodesic ball centered at $o_w$ with  a finite union of geodesic annuli centered at $o_w$ (A + B + C ).   }
\label{figure:anells}
\end{center}
\end{figure}
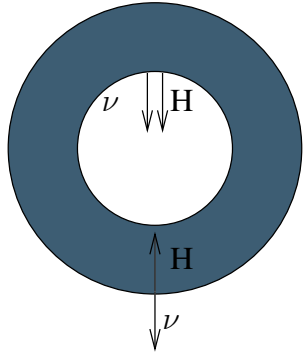
\begin{figure}
\begin{center}
\begin{picture}(0,0)%
\includegraphics{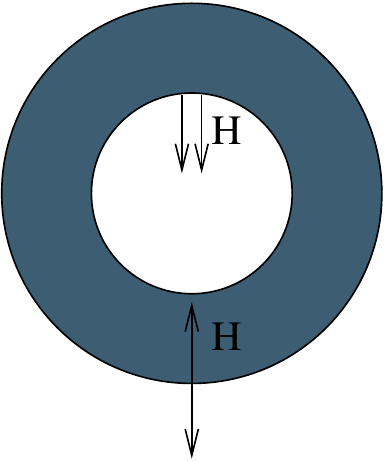}%
\end{picture}%
\setlength{\unitlength}{4144sp}%
\begingroup\makeatletter\ifx\SetFigFont\undefined%
\gdef\SetFigFont#1#2#3#4#5{%
  \reset@font\fontsize{#1}{#2pt}%
  \fontfamily{#3}\fontseries{#4}\fontshape{#5}%
  \selectfont}%
\fi\endgroup%
\begin{picture}(1754,2104)(10374,-1288)
\put(11296,-1141){\makebox(0,0)[lb]{\smash{{\SetFigFont{12}{14.4}{\familydefault}{\mddefault}{\updefault}{\color[rgb]{0,0,0}$\nu$}%
}}}}
\put(10936,164){\makebox(0,0)[lb]{\smash{{\SetFigFont{12}{14.4}{\familydefault}{\mddefault}{\updefault}{\color[rgb]{0,0,0}$\nu$}%
}}}}
\end{picture}%

\caption{The product $h=-\langle \nu, H\rangle$ is not constant in the boundary of an annulus centered at the center of the model.}
\label{figure:anellsno}
\end{center}
\end{figure}

Observe that equality in (\ref{ine-convex}) implies equality in (\ref{ine-mean}) and hence umbilic submanifold.  If moreover we assume $\frac{w'''}{w'}\neq \frac{w''}{w} $  equality in (\ref{ine-convex}) implies equality (\ref{squaone}). If $\Sigma$ has $k>0$ connected components $\Sigma_1,\cdots,\Sigma_k$, each one is a sphere centered at $o_w$. Hence, $\Omega$ is of one of the following types (see figure \ref{figure:anells}):
\begin{enumerate}
\item A geodesic ball centered at $o_w$.
\item A finite union of geodesic annuli centered at $o_w$.
\item A geodesic ball centered at $o_w$ with  a finite union of geodesic annuli centered at $o_w$.
\end{enumerate}

But we can prove that $\Omega$ does not contain annuli. Because, if $\Omega$ contains an annulus since $\eta_w>0$, $H$ is always pointing to $o_w$ (see proposition \ref{propWarpMean}) and the unit normal $\nu$ is always pointing outward to $\Omega$, therefore the product $h=-\langle \nu, H\rangle$ has to be not constant (see figure \ref{figure:anellsno}) in contradiction to the assumption of constant $h$.

Thus finally, $\Omega$ is a ball and $\Sigma$ is a sphere centered at $o_w$.

\end{proof}

\section{Proof of corollary \ref{tone-isoperimetric}}\label{sec:7}
By the Cheeger inequality (see \cite[theorem 3, chapter IV]{Chavel})
\begin{equation}
\lambda_1(\Omega)\geq \frac{1}{4}\left(\inf_{\mathcal{O}\subset\Omega}\frac{\text{A}(\partial \mathcal{O})}{\text{V}(\mathcal{O})}\right)^2,
\end{equation}
where $\mathcal{O}$ ranges on the open subdomains of $\Omega$. By the isoperimetric inequality (\ref{isoperime-ine-model})
\begin{equation}
\begin{aligned}
\lambda_1(\Omega)&\geq \frac{1}{4}\left(\inf_{\mathcal{O}\subset \Omega}\frac{\omega_{n-1}w(\text{Rad}(\text{V}(\mathcal{O}))^{n-1}}{\text{V}(\mathcal{O})}\right)^2\\
&=\frac{1}{4}\left(\inf_{\mathcal{O}\subset \Omega}\mathcal{I}(\text{V}(\mathcal{O}))\right)^2.
\end{aligned}
\end{equation}
Taking into account that $\mathcal{I}$ is non-increasing and $\text{V}(\mathcal{O})\leq \text{V}(\Omega)$, the corollary follows.
\section{Proof of corollary \ref{unuacol}}\label{Proof-Col1}
For any $p>0$, we use the density $d\mu_w$ to define the $L^p_w(\Omega)$-space, by declaring a measurable function $f$ to be an element of $L^p_w(\Omega)$ if the integral
\begin{equation}
\left(\int_{\Omega}\vert f \vert^pd\mu_w\right)^{1/p}
 \end{equation} 
 is finite. The $L^p_{w}(\Omega)$-norm of $f$, $\Vert f \Vert_{L^p_{w}(\Omega)}$, is given by the above expression. H\"older's inequality states that for $p,q>1$ satisfying
 \begin{equation}
 1/p+1/q=1\quad,
 \end{equation}
 one has for $\phi \in \text{L}^p_w(\Omega)$, $\psi \in \text{L}^q_w(\Omega)$ ,
 \begin{equation}\label{Holder}
 \int_{\Omega}\vert \phi\psi\vert d\mu_w\leq \Vert \phi \Vert_{L^p_{w}(\Omega)} \Vert \psi \Vert_{L^q_{w}(\Omega)}\quad .
 \end{equation}
 
 Now, from the Jellett-Minkowski's generalized formula we get
 \begin{equation}
 \sup_{\Omega}\vert \tau_w \vert \sigma_w(\partial \Omega) \geq m \mu_w(\Omega)-\sup_{\Omega}\vert \tau_w \vert \int_\Omega\vert H \vert d\mu_w.
 \end{equation}
 
 Applying the H\"older inequality to the last integral in the above inequality
\begin{equation}\label{after-holder}
 \sup_{\Omega}\vert \tau_w \vert \sigma_w(\partial \Omega) \geq m \mu_w(\Omega)-\sup_{\Omega}\vert \tau_w \vert \left(\int_\Omega\vert H \vert^k d\mu_w\right)^{1/k}\mu_w(\Omega)^{\frac{k-1}{k}}\quad,
 \end{equation}
 
 and the corollary follows.

\section{Proof of corollary \ref{bola}}\label{Proof-Col2}
Applying corollary \ref{unuacol} to the setting of corollary \ref{bola} (namely $w(t)=t$, $d\mu=dV$) we obtain
\begin{equation}\label{iso-quo}
\mathcal{I}_k(\Omega) \geq \frac{m \text{V}(\Omega)^{1/k}}{R}-\left(\int_P\vert H \vert^k dV\right)^{1/k}
\end{equation} 
 Taking into account  the definition of $\mathcal{I}_k$ and the isoperimetric profile, the corollary follows.
 \section{Proof of corollary \ref{bola2}}\label{Proof-Col3}
 Since $P$ has non-shrinking property, there exist $\rho$ such that for any $o\in P$ 
 \begin{equation}
V( B_{\rho}(o))> \left(\frac{R}{m}\right)^k\int_P\vert H\vert^kdV+\epsilon\quad, 
 \end{equation}
 for some $\epsilon>0$.
 
 Therefore, by inequality (\ref{iso-quo}) 
 \begin{equation}
 \mathcal{I}_{k,\rho}(P)>0\quad.
 \end{equation}
   Applying now \cite[Theorem 5 and inequality (15)]{Chavel-mod-constants}  the corollary is proven.
\section{Proof of corollary \ref{cor-minimal}}\label{Proof-Col4}
The first thing to do in order to prove the corollary \ref{cor-minimal} is to study the behavior of the volume of the extrinsic balls. Recall that an extrinsic ball $D_R(o)$ centered to the pole $o\in N$ and with radius $R$ is the sublevel set of the extrinsic distance function. Namely,
\begin{equation}
D_R(o)=\varphi^{-1}\left(B_R^N\left(o\right)\right)\quad,
\end{equation} 
being $B_R^N$ the geodesic ball of $N$ of radius $R$ centered at the pole $o\in N$.
Note that we can construct the order-preserving bijection 
$$
F: \erre^+ \to \erre^+\quad t\to F(t)
$$
given by equation (\ref{order-function}). 

Since $\varphi:P\to N$ is a minimal immersion into a manifold with a pole $N$, applying equation (\ref{laplacian-order-function}) we have
\begin{equation}\label{anterior-2}
\Delta^P F\circ r\geq m w'\circ r\quad .
\end{equation}
Taking into account that $w'>0$, by the maximum principle there exist $R_T$ such that
\begin{equation}
P=D_{R_T}(o)\quad.
\end{equation}

 Now we need the following monotonicity formula

\begin{proposition}
Under the assumptions of corollary \ref{cor-minimal}, the function $f:\erre^+\to\erre^+$ given by
\begin{equation}
f(R):=\frac{\mu_w(D_R)}{w(R)^m}\quad,
\end{equation}
is a nondecreasing function of $R$, and
 \begin{equation}
f(R)\geq c V_m\quad.
\end{equation}
\end{proposition} 
\begin{proof}
Using theorem \ref{Jellet-teo} taking into account that $P$ is minimal and $\partial D_R$ lies in a geodesic sphere of $N$ of radius $R$, we obtain
\begin{equation} \label{equ1}
m\mu_w(D_R)\leq \frac{1}{\eta_w(R)}\sigma_w(\partial D_R)\quad .
\end{equation}
By the coarea formula we get
\begin{equation}\label{equ2}
\sigma_w(\partial D_R)\leq \frac{d}{dR}\mu_w(D_R)\quad.
\end{equation}
Therefore, using inequalities (\ref{equ1}) and (\ref{equ2}) together
\begin{equation}
\frac{d}{dR}\ln\left(\mu_w\left(D_R\right)\right)\geq \frac{d}{dR}\ln\left(w\left(R\right)^m\right)\quad .
\end{equation}
Hence, we obtain the desired monotonicity formula. Observe also that 
\begin{equation}
\lim_{R\to 0}\frac{\mu_w(D_R)}{w(R)^m}\geq \lim_{R\to 0}c\frac{V(D_R)}{w(R)^m}\geq c V_m\quad.
\end{equation}
And the proposition follows.\end{proof}

On the other hand, from inequality (\ref{equ1})
\begin{equation}
\sigma_w(\partial D_R)\geq m\mu_w(D_R)\eta_w(R)\quad.
\end{equation}
But taking into account the definition of the extrinsic ball and $w$-weighted area, and using the above proposition
\begin{equation}
\begin{aligned}
A(\partial D_R)\geq&\frac{m}{w(R)}\mu_w(D_R)=m\left(\frac{\mu_w(D_R)}{w(R)^m}\right)^{\frac{1}{m}}\mu_w(D_R)^{1-1/m}\\
\geq&m\left(cV_m\right)^{\frac{1}{m}}\mu_w(D_R)^{1-1/m}\quad.
\end{aligned}
\end{equation}
Hence, finally the corollary follows changing $R$ by $R_T$ in the above inequality.

\section{Proof of corollary \ref{cor-volume}}\label{Proof-Col5}
Applying inequality (\ref{after-holder}) to the extrinsic ball $D_R$ taking into account that $\sup_{D_R}\vert \tau_w\vert\leq \frac{1}{\sqrt{-b}}$ we get
\begin{equation}
 m \mu_w(D_R) \leq  \frac{\sigma_w(\partial D_R)}{\sqrt{-b}}+\frac{1}{\sqrt{-b}} \left(\int_{D_R}\vert H \vert^p d\mu_w\right)^{1/p}\mu_w(D_R)^{\frac{p-1}{p}}\quad,
\end{equation}
 by inequality (\ref{equ2}),
\begin{equation}
m\sqrt{-b}-\Vert H \Vert_{L_w^p(P)}\mu_w(D_R)^\frac{-1}{p}\leq \frac{d}{dR}\ln\mu_w(D_R)\quad,
\end{equation}
Given $R_0>0$, for any $R\geq R_0$
\begin{equation}
m\sqrt{-b}-\Vert H \Vert_{L_w^p(P)}\mu_w(D_{R_0})^\frac{-1}{p}\leq \frac{d}{dR}\ln\mu_w(D_R)\quad,
\end{equation}

If the submanifold has finite volume, there exist a divergent sequence $\{R_i\}_{i=1}^\infty$ such that

\begin{equation}
m\sqrt{-b}-\Vert H \Vert_{L_w^p(P)}\mu_w(D_{R_0})^\frac{-1}{p}\leq \limsup \frac{d}{dR}\ln\mu_w(D_{R_i})=0.
\end{equation}
And therefore the corollary follows letting $R_0$ tend to infinity.
\section*{Acknowledgments}
The author is very grateful to Simon Brendle for his useful discussion about his paper \cite{Brendle} and to Antonio Cañete  for show me the paper \cite{Mau-Mor-2009}.
\def\cprime{$'$} \def\cprime{$'$} \def\cprime{$'$} \def\cprime{$'$}
  \def\cprime{$'$}
\providecommand{\bysame}{\leavevmode\hbox to3em{\hrulefill}\thinspace}
\providecommand{\MR}{\relax\ifhmode\unskip\space\fi MR }
\providecommand{\MRhref}[2]{%
  \href{http://www.ams.org/mathscinet-getitem?mr=#1}{#2}
}
\providecommand{\href}[2]{#2}

\end{document}